\documentclass[12pt,reqno]{amsart}
\usepackage{amsfonts}
\usepackage{mathtools}
\usepackage{nicefrac}
\usepackage[left=1.4in,right=1.4in,top=1.2in,bottom=1.2in]{geometry}
\usepackage[msc-links, lite]{amsrefs}

\usepackage{enumitem}
\setlist[enumerate,itemize]{itemsep=0.5ex}

\usepackage{hyperref}
\usepackage[capitalise]{cleveref}

\theoremstyle{plain}
\newtheorem{theorem}{Theorem}[section]
\newtheorem{proposition}[theorem]{Proposition}
\newtheorem{lemma}[theorem]{Lemma}

\newtheorem{thmx}{Theorem}

\theoremstyle{definition}

\theoremstyle{remark} 
\newtheorem{remark}[theorem]{Remark}

\numberwithin{equation}{section}

\newcommand{\interior}[1]{{\kern0pt#1}^{\mathrm{o}}}
\newcommand{\sph}{\mathbb{S}}

\DeclareMathOperator{\R}{\mathbb{R}}
\DeclareMathOperator{\Z}{\mathbb{Z}}
\DeclareMathOperator{\Lip}{Lip}
\DeclareMathOperator{\vol}{vol}
\DeclareMathOperator{\scal}{Sc}
\DeclareMathOperator{\Ric}{Ric}
\DeclareMathOperator{\mean}{H}
\DeclareMathOperator{\II}{A}
\DeclareMathOperator{\pr}{pr}

\usepackage{todonotes}

\makeatletter
\providecommand\@dotsep{5}
\renewcommand{\listoftodos}[1][\@todonotes@todolistname]{%
  \@starttoc{tdo}{#1}}
\makeatother

\begin{document}

\title{scalar curvature rigidity of the four-dimensional sphere}

\author{Simone Cecchini}
\address[Simone Cecchini]{Department of Mathematics, Texas A\&M University}
\email{cecchini@tamu.edu}

\author{Jinmin Wang}
\address[Jinmin Wang]{Department of Mathematics, Texas A\&M University}
\email{jinmin@tamu.edu}
\thanks{The second author is partially supported by  NSF  1952693, 2247322 and NSFC 12171095.}
 
\author{Zhizhang Xie}
\address[Zhizhang Xie]{ Department of Mathematics, Texas A\&M University }
\email{xie@tamu.edu}
\thanks{The third author is partially supported by NSF  1952693 and 2247322.}

\author{Bo Zhu}
\address[Bo Zhu]{ Department of Mathematics, Texas A\&M University }
\email{bozhu@tamu.edu}
\thanks{The fourth author is partially supported by NSF  1952693 and 2247322.}

\begin{abstract}
Let $(M,g)$ be a four-dimensional closed connected oriented (possibly non-spin) Riemannian manifold with scalar curvature bounded below by $n(n-1)$. We prove that, if $f$ is a smooth distance non\nobreakdash-increasing map of non\nobreakdash-zero degree from $(M, g)$ to the unit four-sphere, then $f$ is an isometry.
Following ideas of Gromov, we utilize $\mu$\nobreakdash-bubbles and a version with coefficients of the rigidity of the three-sphere to rule out the case where all the inequalities are strict.
Our proof of rigidity exploits monotonicity results for the harmonic map heat flow coupled with the Ricci flow due to Lee and Tam.
\end{abstract}
\maketitle

\section{Introduction}
Extremality and rigidity properties of Riemannian manifolds with lower scalar curvature bounds have been the subject of intensive study in recent years.
For a comprehensive overview of the subject, we refer to Gromov’s \emph{Four lectures on scalar curvature}~\cite{gromov-four-lectures}.
A cornerstone result in comparison geometry with scalar curvature is the rigidity of the round sphere in the spin setting, established by Llarull.
Throughout this paper, we denote by $g_{\sph^n}$ the standard round metric on the $n$\nobreakdash-dimensional sphere $\sph^n$.

\begin{theorem}[{\cite[Theorem~B]{Llarull}}]\label{thm:Llarull}
    Let $(M,g)$ be an $n$-dimensional closed connected spin Riemannian manifold with $\scal_g\geq n(n-1)$.
    If $f\colon(M,g)\to(\sph^n,g_{\sph^n})$ is a smooth, distance non\nobreakdash-increasing map of non\nobreakdash-zero degree,
    then $f$ is an isometry.
\end{theorem}

This result illustrates the beautiful interplay between metric, curvature and topological information in scalar curvature geometry.
Its proof relies on the Dirac operator method, requiring the hypothesis that $M$ is spin.
A big open question in the field is whether the spin assumption can be dispensed with in \cref{thm:Llarull}.
In this paper, we address this question affirmatively, at least in dimension four.

\begin{thmx}\label{thm:sphere}
    Let $(M,g)$ be an four-dimensional closed connected oriented (possibly non-spin) Riemannian manifold with $\scal_g\geq12$.
    If $f\colon(M,g)\to(\sph^n,g_{\sph^n})$ is a smooth, distance non\nobreakdash-increasing map of non\nobreakdash-zero degree,
    then $f$ is an isometry.
\end{thmx}

Our approach intertwines various techniques from geometric analysis: minimal hypersurfaces, Ricci flow, and harmonic map heat flow.
A key tool in our method is the utilization of $\mu$\nobreakdash-bubbles, that are stable solutions to prescribed mean curvature problems.
This technique, pioneered by Gromov~\cite[Section~5\nicefrac{5}{6}]{Gromov-macroscopic-dimension}, has been successfully used in addressing some challenging questions in scalar curvature geometry.
Examples of its applications can be found in~\cites{li-polyhedron, chodosh2023generalized, gromov2020metrics, zhu_width_mu, chodosh-sufficiently-connected, cecchini-two-ends}.
Drawing from Gromov's ideas~\cite[Section~5]{gromov-four-lectures}, we utilize this technique to rule out the case where all the inequalities in \cref{thm:sphere} are strict.
However, directly proving extremality and rigidity using $\mu$-bubbles poses a significant challenge.
To overcome this difficulty, we exploit the monotonicity of the harmonic map heat flow coupled with the Ricci flow, recently established by Lee and Tam~\cite{LeeTam}.
This approach enables us to reduce \cref{thm:sphere} to the situation where all the inequalities are strict, except when the metric is Einstein.
Remarkably, Llarull's rigidity theorem for Einstein manifolds follows from classical comparison geometry.

To further illustrate our strategy, let us compare it with Gromov's perspective~\cite[Section~5.7]{gromov-four-lectures}.
We regard $\sph^n$ with two antipodal points removed as a warped product over $\sph^{n-1}$.
From this viewpoint, \cref{thm:Llarull} becomes a question about the scalar curvature rigidity of the degenerate spherical band 
\begin{equation}\label{eq:degenerate_band}
    \left(\sph^{n-1}\times(-\pi/2,\pi/2),\cos^2(t)g_{\sph^{n-1}}+dt^2\right).
\end{equation}
In the spin setting, this question has been recently addressed independently by B\"ar-Brendle-Hanke-Wang~\cite{Baer:2023aa} and by the second and third authors~\cite{wang2023scalar} utilizing the Dirac operator method.
Gromov outlined an alternative approach to this problem employing $\mu$\nobreakdash-bubbles.
The crucial observation~\cite[Section~5.5]{gromov-four-lectures} is that the variational formulas for $\mu$\nobreakdash-bubbles, in this context, are related to the following stronger version of \cref{thm:Llarull}, due to Listing. 

\begin{theorem}[{\cite[Theorem~1]{Listing:2010te}}]\label{thm:Listing}
    Let $(M,g)$ be an $n$-dimensional closed connected spin Riemannian manifold.
    If $f\colon(M,g)\to(\sph^n,g_{\sph^n})$ is a smooth map of non\nobreakdash-zero degree such that $\scal_g(p)\geq n(n-1)\left\|df_p\right\|^2$ for any $p\in M$,
    then there exists a constant $c>0$ such that $f\colon(M,c\cdot g)\to(\sph^n,g_{\sph^n})$ is an isometry.  
\end{theorem}
\noindent Here $\left\|df_p\right\|$ denotes the operator norm of the linear map $df_p$, see the discussion before \cref{lem:separating_hypersurface}.
For the proof of this theorem, we also refer to~\cite[Theorem 3.3]{wang2023scalar}.
Since all three-dimensional oriented manifolds are spin, this observation enables us to use $\mu$\nobreakdash-bubbles to study extremality and rigidity properties of spherical bands in dimension four without the spin condition.

Employing the outlined strategy to study the degenerate spherical band~\eqref{eq:degenerate_band} poses two main challenges.
Firstly, constructing suitable $\mu$\nobreakdash-bubbles on open incomplete manifolds presents significant difficulties.
While in dimension three this construction has been successfully carried out by Hu, Liu, and Shi~\cite{hu-3d-spherical} leveraging the extra control provided by the Gauss\nobreakdash-Bonnet theorem, how to extend this approach to higher dimensions remains unclear.
The authors intend to tackle this issue in future work.
Secondly, it is implied in~\cite[Section~5.5]{gromov-four-lectures} that the same construction is expected to carry over to degenerate bands over the torus, provided that the warping function is log-concave.
However, \cite[Example~4.1]{wang2023scalar} shows that this condition is not sufficient for the rigidity of torical bands and, \emph{a fortiori}, for the existence of a $\mu$-bubble with the desired properties.
The authors also plan to investigate general conditions for the existence of $\mu$\nobreakdash-bubbles in open incomplete bands in future work.

Instead of focusing on degenerate spherical bands, this paper adopts a different strategy to establish \cref{thm:sphere}, as outlined below.
\begin{enumerate}
    \item We first rule out from \cref{thm:sphere} the case when all the inequalities are strict, following Gromov's approach.
    This involves utilizing $\mu$\nobreakdash-bubbles and \cref{thm:Listing} in dimension three. 
    More precisely, we make use of these techniques to prove a comparison theorem with scalar and mean curvature bounds for compact spherical bands.
    Our comparison results are in the same spirit as~\cites{cecchini2022scalar,Raede23}.
    \item We then employ the harmonic map heat flow coupled with the Ricci flow to demonstrate that the general case of \cref{thm:sphere} reduces to the situation where all the inequalities are strict, unless the metric $g$ is Einstein with $\Ric_g=3g$. 
    Here, we make use of recent results of Lee and Tam~\cite{LeeTam}, showing that the harmonic map heat flow coupled with the Ricci flow provides appropriate control of the Lipschitz constant with respect to the change of  the scalar curvature under Ricci flow.  
    \item Finally, we prove \cref{thm:sphere} for Einstein manifolds, which follows as a consequence of Bishop's volume comparison theorem.
\end{enumerate}
This approach offers a novel perspective on proving \cref{thm:sphere}, circumventing the difficulties associated with degenerate spherical bands and exploiting powerful tools from geometric analysis.

The remainder of this paper is organized as follows.
\cref{sec:mu-bubbles} reviews relevant results on the existence and properties of $\mu$-bubbles on compact Riemannian bands.
In \cref{sec:scalar_mean}, we prove a comparison result for compact spherical bands with scalar and mean curvature bound.
In \cref{sec:Ricci_flow_Einstein}, we use the harmonic map heat flow coupled with the Ricci flow to show that Theorem \ref{thm:sphere} can be reduced to the case where all inequalities become strict unless the metric $g$ is Einstein.
Finally, in \cref{sec:4-sphere_rigidity} we show that the comparison theorem from \cref{sec:scalar_mean} suffices to rule out strict inequalities from \cref{thm:sphere} and that \cref{thm:sphere} holds for Einstein metrics. This completes the proof of \cref{thm:sphere}.

\subsection*{Acknowledgments.} We would like to thank Alessandro Carlotto, Man-Chun Lee, and Thomas Schick for helpful discussions.

\section{$\mu$-bubbles with mean curvature bound}\label{sec:mu-bubbles}
We review the properties of $\mu$\nobreakdash-bubbles to be utilized in this paper.
For more details, we refer the reader to the work of Gromov~\cite{gromov-four-lectures} and Zhu~\cite{zhu_width_mu}.

Let us start by establishing notation and conventions.
Let $(X,g)$ be an oriented Riemannian manifold.
We stress that, throughout this paper, all manifolds are oriented and connected.
We denote the Ricci curvature tensor by $\Ric_g$, and the scalar curvature by $\scal_g$.
The Riemannian volume form is denoted by $dV_g$. 
For an embedded hypersurface $Z\subset X$, $g_Z$ stands for the restriction of $g$ to $Z$. 
If the boundary $\partial X$ of $X$ is non-empty, $\II_g(\partial X)$ denotes the second fundamental form with respect to the inward unit normal field $\nu$ along $\partial X$, and $\mean_g(\partial X)$ its trace, the mean curvature with respect to $\nu$. 
As per our convention, the boundary of the closed unit ball in $\R^n$ has mean curvature equal to $n-1$.

A \emph{band} is a compact connected manifold with boundary \(V\) together with a decomposition \(\partial V = \partial_- V \sqcup \partial_+ V\), where \(\partial_\pm V\) are non-empty unions of components.
A \emph{proper separating hypersurface} for a band $V$ is a closed embedded hypersurface $\Sigma\subset \interior{V}$ such that no connected component of $V\setminus\Sigma$ contains a path $\gamma\colon [0,1]\to V$ with $\gamma(0)\in\partial_-V$ and $\gamma(1)\in\partial_+V$.
Note that a proper separating hypersurface $\Sigma$ for an orientable band $V$ is also orientable.
Additionally, if $(V,g)$ is a Riemannian band and $\Sigma$ is a proper separating hypersurface for $V$, we denote by $(\hat V,g)$ the Riemannian band isometrically embedded in $(V,g)$ such that $\partial_-\hat V=\partial_-V$ and $\partial_+\hat V=\Sigma$. 
We adopt the convention of denoting by $\II_g(\Sigma)$ and $\mean_g(\Sigma)$ respectively the second fundamental form and the mean curvature on $\Sigma$ regarded as a union of components of $\partial\hat V$.

Let us now turn to $\mu$-bubbles. 
Let $(V,g)$ be an $n$-dimensional  Riemannian band. We denote by $\mathcal C(V)$ the set of all Caccioppoli sets $\hat\Omega$ in $V$ such that $\hat\Omega$ contains an open neighborhood of $\partial_-V$ and $\hat\Omega\cap\partial_+V=\emptyset$. 
For the notion of Caccioppoli sets and their properties, we refer the reader to~\cite{giusti-minimal-surfaces}. 
If $\hat\Omega$ is a smooth Caccioppoli set in $\mathcal C(V)$, then $\partial\hat\Omega\cap\interior V$ is a separating hypersurface for $V$.
For any given smooth function $\mu$ on $V$,  we consider the functional
\begin{equation*}
    \mathcal A_\mu(\hat\Omega)=\mathcal H^{n-1}(\partial^\ast\hat\Omega\cap\interior V)-\int_{\hat\Omega}\mu\,d\mathcal H^n
\end{equation*}
where $\mathcal H^k$ denotes the $k$-dimensional Hausdorff measure on $(V,g)$, and  $\partial^\ast\hat\Omega$ denotes the reduced boundary of $\hat\Omega$.
We say that a Caccioppoli set $\Omega\in\mathcal C(V)$ is a \emph{$\mu$-bubble} if it minimizes the functional $\mathcal A_\mu$, that is, if $\Omega$ satisfies
\[
    \mathcal A_\mu(\Omega)=\inf\left\{\mathcal A_\mu(\hat\Omega)\colon\hat\Omega\in\mathcal C(V)\right\}.
\]

\noindent The next lemma states the existence of $\mu$\nobreakdash-bubbles under strict mean curvature bound.

\begin{lemma}[{\cite[Section 2]{zhu_width_mu}, \cite[Lemma~4.2]{Raede23}}]\label{lem:mu_bubble_existence}
    For $n\leq7$, let $(V,g)$ be an $n$-dimensional Riemannian band and $\mu\colon V\to\R$ a smooth function such that $\mean_g(\partial_\pm V)>\pm\mu|_{\partial_\pm V}$.
    Then there exists a smooth $\mu$-bubble $\Omega$.
\end{lemma}

\noindent In the next lemma we collect the properties of smooth $\mu$\nobreakdash-bubbles that will be used in this paper.

\begin{lemma}[{\cite[Section 2]{zhu_width_mu},\cite[Lemmas~4.3 and 4.4]{Raede23}}]\label{lem:mu_bubble_variation}
    Let $(V,g)$ be an $n$-dimensional Riemannian band, let $\mu\colon V\to\R$ be a smooth function, and let $\Omega$ be a smooth $\mu$-bubble.
    Then $\Sigma\coloneqq\partial\Omega\cap\interior V$ is a proper separating hypersurface for $V$ satisfying
    \begin{equation}\label{eq:first_variation}
        \int_\Sigma\left(\mean_g(\Sigma)-\mu\right)u\,dV_{g_\Sigma}=0       
    \end{equation}
    and
    \begin{multline}\label{eq:second_variation}
        \int_\Sigma\left|\nabla_{g_\Sigma}u\right|^2+\frac{1}{2}\scal_{g_\Sigma}u^2\,dV_{g_\Sigma}\\
        \geq\frac{1}{2}\int_\Sigma \left(\scal_g-\mean_g^2(\Sigma)+\left|\II_g(\Sigma)\right|^2+2\mean_g(\Sigma)\mu+2g(\nabla_g\mu,\nu)\right)u^2\,dV_{g_\Sigma}
    \end{multline}
    for every $u\in C^\infty(\Sigma)$.
\end{lemma}

\begin{remark}
    Since the $\mu$-bubble $\Omega$ minimizes the functional $\mathcal A_\mu$, Equation~\eqref{eq:first_variation} follows from the first variation formula, whereas Inequality~\eqref{eq:second_variation} follows from the second variation formula. See \cite[Section 2]{zhu_width_mu} and also ~\cite[Lemmas~4.3 and 4.4]{Raede23}.
\end{remark}

\section{scalar-mean bounds on spherical bands}\label{sec:scalar_mean}
We prove a comparison theorem for spherical bands with scalar and mean curvature bounds.
Let $I=[t_-,t_+]$ be a compact interval, and let $\phi\colon I\to\R_+$ be a smooth function.
Consider the four-dimensional band $\sph^3_I\coloneqq\sph^3\times I$, equipped with the warped product metric
\[
    g_\phi(p,t)\coloneqq\phi^2(t)g_{\sph^3}(p)+dt^2,\qquad\forall(p,t)\in\sph^3\times I=\sph^3_I.
\]
Let $h_\phi\colon I\to\R$ be the smooth function defined as
\begin{equation}\label{eq:h_phi}
    h_\phi(t)\coloneqq3\frac{d}{dt}\log(\phi(t))=3\frac{\phi^\prime(t)}{\phi(t)},\qquad\forall t\in I.   
\end{equation}
Utilizing the classical formula for the scalar curvature of warped product metrics~\cite[Ch.IV, Formula~(6.15)]{spingeometry}, 
\begin{equation}\label{eq:g_phi_scal}
    \scal_{g_\phi}(p,t)
    =\frac{6}{\phi^2(t)}-\frac{4}{3}h_\phi^2(t)-2h_\phi^\prime(t),\qquad\forall (p,t)\in \sph^3\times I.
\end{equation}
Additionally, $\partial_\pm\sph^3_I=\sph_3\times\{t_\pm\}$ and
\begin{equation}\label{eq:g_phi_mean}
    H_{g_\phi}(\partial_\pm \sph^3_I)=\pm h_\phi(t_\pm).
\end{equation}
We say that a positive smooth function $\phi(t)$ is log-concave if $\log(\phi(t))$ is a concave function.

\begin{proposition}\label{prop:scalar-mean_strict}
    Let $(V,g)$ be a four-dimensional oriented Riemannian band.
    Let $I=[t_-,t_+]$ be a compact interval, and $\phi\colon I\to\R_+$ be a smooth log\nobreakdash-concave function.
    Let $f\colon(V,g)\to(\sph^3_I,g_\phi)$ be a smooth map.
    Suppose
    \begin{enumerate}[label=\textup{(\roman*)}]
        \item $f$ is distance non\nobreakdash-increasing,\label{item:scalar-mean_strict1}
        \item $f(\partial_\pm V)\subseteq\partial_\pm\sph^3_I$\label{item:scalar-mean_strict2},
        \item $\scal_g> f^\ast\scal_{g_\phi}$,\label{item:scalar-mean_strict3}
        \item $\mean_g(\partial_\pm V)>f^\ast\mean_{g_\phi}(\partial_\pm\sph^3_I)$.\label{item:scalar-mean_strict4}
    \end{enumerate}
    Then $f$ has degree zero.
\end{proposition}

To prove this proposition, we will construct a separating hypersurface for $V$ with suitable properties forcing the degree of $f$ to be zero.
Let us first introduce some additional notation.
Let $\pr_{\sph^3}\colon \sph^3_I=\sph^3\times I\to\sph^3$ be the projection onto the first factor.
Under the hypotheses of \cref{prop:scalar-mean_strict}, let $Z\subset V$ be a closed hypersurface.
Consider the smooth map
\begin{equation}\label{eq:f_Z}
    f_Z\coloneqq\pr_{\sph^3}\circ f\circ i_Z\colon Z\to\sph^3,
\end{equation}
where $i_Z\colon Z\hookrightarrow V$ denotes the inclusion of $Z$ into $V$.
The next lemma contains well\nobreakdash-known properties of the map $f_Z$, see~\cite[Lemma~6.3]{Raede23}.
We include a brief proof for clarity.

\begin{lemma}\label{lem:f_Sigma_degree}
    Let $V$ be a four-dimensional oriented band, let $I$ be a compact interval, and let $f\colon V\to \sph^3_I$ be a smooth map of non\nobreakdash-zero degree such that $f(\partial_\pm V)\subseteq\partial_\pm\sph^3_I$.
    Suppose $\Sigma$ is a proper separating hypersurface of $V$.
    Then there exists a connected component $Y$ of $\Sigma$ such that $f_Y\colon Y\to\sph^3$ has non\nobreakdash-zero degree.    
\end{lemma}

\begin{proof}
    Let $\alpha$ be a generator of $H^1(\sph^3_I,\partial\sph^3_I;\Z)=\Z$, and let $[\sph^3_I,\partial\sph^3_I]\in H_4(\sph^3_I,\partial\sph^3_I;\Z)$ be the fundamental class.
    Note that 
    \begin{equation}\label{eq:f_Sigma_degree1}
        (\pr_{\sph^3})_\ast([\sph^3_I,\partial\sph^3_I]\frown\alpha)=[\sph^3]
    \end{equation}
    as the Lefschetz dual of $\alpha$ is represented by any submanifold $\sph^3\times\{t\}$, for $t\in I$.
    Since $f(\partial_\pm V)\subseteq\partial_\pm\sph^3_I$ and since $\Sigma$ is a proper separating hypersurface for $V$, there exists a union of components of $\Sigma$, that we denote by $\Sigma^\prime$, such that $[\Sigma^\prime]$ is Lefschetz dual to $f^\ast(\alpha)$.
    Hence, 
    \begin{equation}\label{eq:f_Sigma_degree2}
        f_\ast([\Sigma^\prime])=f_\ast([V,\partial V]\frown f^\ast(\alpha))
        =f_\ast([V,\partial V])\frown\alpha=d[\sph^3_I,\partial\sph^3_I]\frown\alpha
    \end{equation}
    where $d=\deg(f)$ and $[V,\partial V]\in H_4(V,\partial V;\Z)$ is the fundamental class.
    From \eqref{eq:f_Sigma_degree1} and \eqref{eq:f_Sigma_degree2}, we deduce that $(\pr_{\sph^3}\circ f)_\ast([\Sigma^\prime])=d[\sph^3]$.
    Therefore, if $f_Y$ has degree zero for every component $Y$ of $\Sigma$, we must have $d=0$.
\end{proof}

For a smooth map $f\colon (X,g)\to (Y,h)$ of Riemannian manifolds, we denote by $\left\|df_p\right\|$ the operator norm of the linear map $df_p\colon (T_pX,g_p)\to (T_{f(p)}Y,h_p)$, that is, $\left\|df_p\right\|$ is the infimum of the numbers $c>0$ such that $h_{f(p)}(df_pv,df_pv)^{1/2}\leq cg_p(v,v)^{1/2}$, for all $v\in T_pX$.
We denote by $\left\| df\right\|\colon X\to\R_{\geq 0}$ the function whose value at a point $p\in X$ is $\left\|df_p\right\|$.
Note that $f$ is distance non-increasing if and only if $\left\|df_p\right\|\leq1$ for all $p\in X$.

\begin{lemma}\label{lem:separating_hypersurface}
    Assume the same hypotheses as in \cref{prop:scalar-mean_strict}.
    Then there exists a proper separating hypersurface $\Sigma$ for $V$ such that
    \begin{equation}\label{eq:separating_hypersurface}
        \int_\Sigma\left|\nabla_{g_\Sigma}u\right|^2
        +\frac{1}{2}\left(\scal_{g_\Sigma}
        -6\left\|df_\Sigma\right\|^2\right)u^2\,dV_{g_\Sigma}>0
    \end{equation}
    for all non\nobreakdash-zero $u\in C^\infty(\Sigma)$.
\end{lemma}

\begin{proof}
Let $\pr_I\colon \sph^3_I=\sph^3\times I\to I$ be the projection onto the second factor.
Consider the smooth map
\[
    \mu_\phi\coloneqq h_\phi\circ\pr_I\circ f\colon V\to\R,
\]
where $h_\phi$ is the function defined by~\eqref{eq:h_phi}.
By Condition~\ref{item:scalar-mean_strict4} and Equation~\eqref{eq:g_phi_mean}, $\mean_g(\partial_\pm V)>f^\ast\mean_{g_\phi}(\partial_\pm\sph^3_I)=\pm\mu_\phi|_{\partial_\pm V}$.
By \cref{lem:mu_bubble_existence} and \cref{lem:mu_bubble_variation}, there exists a smooth $\mu_\phi$-bubble $\Omega$ for $V$, and $\Sigma\coloneqq \partial\Omega\cap\interior V$ is a closed separating hypersurface for $V$ satisfying~\eqref{eq:first_variation} and ~\eqref{eq:second_variation}. By~\eqref{eq:first_variation}, $\mean_g(\Sigma)=\mu_\phi$.
    Since $\left|\II_g(\Sigma)\right|^2\geq\frac{\mean_g^2(\Sigma)}{n-1}$,
    \[
        -\mean_g(\Sigma)^2+\left|\II_g(\Sigma)\right|^2+2\mean_g(\Sigma)\mu_\phi\geq \mu_\phi^2+\frac{\mu_\phi^2}{n-1}=\frac{n}{n-1}\mu_\phi^2.
    \]
     From Inequality~\eqref{eq:second_variation}, we deduce that    
\begin{equation}\label{eq:separating_hypersurface1}
        \int_\Sigma\left|\nabla_{g_\Sigma} u\right|^2+\frac{1}{2}\scal_{g_\Sigma}u^2\,d V_{g_\Sigma}
        \geq\frac{1}{2}\int_\Sigma\left(\scal_g+\frac{4}{3}\mu_\phi^2+2g(\nabla_g\mu_\phi,\nu)\right)u^2\,d V_{g_\Sigma}
\end{equation}
for all $u\in C^\infty(\Sigma)$.
Let $p\in\Sigma$.
By Condition~\ref{item:scalar-mean_strict1}, $|\nabla_g (\pr_I\circ f)(p)|\leq1$.
Since $\phi$ is log-concave, $h_\phi^\prime(\pr_I(f(p)))\leq0$.
Therefore, 
\begin{equation*}
    h_\phi^\prime(\pr_I(f(p)))
    \leq h_\phi^\prime(\pr_I(f(p)))\left|\nabla_g (\pr_I\circ f)(p)\right|\leq g(\nabla_g\mu_\phi(p),\nu(p)).
\end{equation*}
Thus,
\begin{align}\label{eq:separating_hypersurface2}
    \frac{4}{3}\mu_\phi^2(p)+2g(\nabla_g\mu_\phi(p),\nu(p))
    \geq& \frac{4}{3}h_\phi^2(\pr_I(f(p)))+2h_\phi^\prime(\pr_I(f(p)))\\\notag
    =&\frac{6}{\phi^2(\pr_I(f(p)))}-\scal_{g_\phi}(f(p))
\end{align}
where in the last equality we used~\eqref{eq:g_phi_scal}.
Using the chain rule,
\begin{equation}\label{eq:separating_hypersurface3}
    \left\|(df_\Sigma)_p\right\|^2\leq\frac{1}{\phi^2(\pr_I(f(p)))}.
\end{equation}
Using Condition~\ref{item:scalar-mean_strict3} and Inequalities~\eqref{eq:separating_hypersurface2} and~\eqref{eq:separating_hypersurface3}, we deduce
\begin{equation*}
    \scal_g(p)+\frac{4}{3}\mu_\phi^2(p)+2g\left(\nabla_g\mu_\phi(p),\nu(p)\right)
    >6\left\|(df_\Sigma)_p\right\|^2,\qquad\forall p\in\Sigma.
\end{equation*}
Finally, from the previous inequality and Inequality~\eqref{eq:separating_hypersurface1} we conclude that~\eqref{eq:separating_hypersurface} holds for every non\nobreakdash-zero $u\in C^\infty(\Sigma)$.
\end{proof}

Under the hypotheses of \cref{prop:scalar-mean_strict}, let $Z\subset V$ be a closed hypersurface.
Consider the second-order formally self-adjoint elliptic differential operator
\begin{equation}\label{eq:L_g}
    \mathcal L_{g_Z}
    \coloneqq -\Delta_{g_Z}+\frac{1}{8}\left(\scal_{g_Z}-6\left\|d f_Z\right\|^2\right)
\end{equation}
where $\Delta_{g_Z}$ denotes the Laplace-Beltrami operator of $(Z,g_Z)$ and $f_Z\colon Z\to\sph^3$ is defined by \eqref{eq:f_Z}.
Note that we adopt the sign convention for $\Delta_{g_Z}$ such that 
\[
    -\int_Z\left(\Delta_{g_z}u\right) u\,dV_{g_Z}=\int_Z\left|\nabla_{g_Z}u\right|^2\,dV_{g_Z},\qquad\forall u\in C^\infty(Z).
\]
With this convention, $-\Delta_{g_Z}$ has nonnegative spectrum.
By classical elliptic theory, the spectrum of $\mathcal L_{g_Z}$ is a discrete set bounded from below consisting only of eigenvalues.
Moreover, the eigenfunctions relative to each eigenvalue are smooth.

\begin{lemma}\label{lem:positive_eigenvalues}
    Suppose that
    \begin{equation}\label{eq:positive_eigenvalues}
        \int_Z\left|\nabla_{g_Z}u\right|^2
        +\frac{1}{2}\left(\scal_{g_Z}
        -6\left\|df_Z\right\|^2\right)u^2\,dV_{g_Z}>0
    \end{equation}
    for all non\nobreakdash-zero $u\in C^\infty(Z)$.
    Then the lowest eigenvalue of $\mathcal L_{g_Z}$ is positive.
\end{lemma}

\begin{proof}
    Let $\lambda\in\R$ be an eigenvalue of $\mathcal L_{g_Z}$.
    We will show that $\lambda>0$.
    Let $u\in C^\infty(Z)$ be an eigenfunction corresponding to $\lambda$, that is,
    \[
        \Delta_{g_Z}u+\lambda u=\frac{1}{8}\left(\scal_{g_Z}-6\left\|d f_Z\right\|^2\right)u
    \]
    with $u\not\equiv 0$.
    We have
    \begin{align*}
        \int_Z\left|\nabla_{g_Z}u\right|^2-\lambda u^2\,dV_{g_Z}
        &=-\int_Z\left(\Delta_{g_Z}u+\lambda u\right) u\,dV_{g_Z}\\
        &=-\frac{1}{8}\int_Z\left(\scal_{g_Z}-6\left\|d f_Z\right\|^2\right)u^2\,dV_{g_Z}\\
        &<\frac{1}{4}\int_Z\left|\nabla_{g_Z}u\right|^2\,dV_{g_Z},
    \end{align*}
    where in the last inequality we used~\eqref{eq:positive_eigenvalues}.
    Therefore,
    \begin{equation*}
        \lambda\int_Zu^2\,dV_{g_Z}
        > \frac{3}{4}\int_Z\left|\nabla_{g_Z}u\right|^2\,dV_{g_Z}.
    \end{equation*}
    Since $u\not\equiv 0$, we conclude that $\lambda>0$.
\end{proof}

To prove \cref{prop:scalar-mean_strict}, let us specialize \cref{thm:Listing} to the three-dimensional case.
Since all three-dimensional oriented manifolds are spin, we obtain:

\begin{proposition}\label{prop:Listing_3D}
    Let $(M,g)$ be a three-dimensional closed connected oriented Riemannian manifold.
    If $f\colon (M,g)\to(\sph^3,g_{\sph_3})$ is a smooth map of non\nobreakdash-zero degree such that $\scal_g(p)\geq 6\left\|df_p\right\|^2$ for all $p\in M$, then there exists a constant $c>0$ such that $f\colon(M,c\cdot g)\to (\sph^3,g_{\sph^3})$ is an isometry.
\end{proposition}

\begin{remark}\label{rem:strict-Listing-3D}
    Let $(M,g)$ be a three-dimensional closed connected oriented Riemannian manifold.
    \cref{prop:Listing_3D} implies that there are no smooth maps $f\colon (M,g)\to(\sph^3,g_{\sph_3})$ of non\nobreakdash-zero degree such that $\scal_g(p)>6\left\|df_p\right\|^2$ for all $p\in M$. In fact, we will only use this consequence in the proof of \cref{thm:sphere}.
\end{remark}
We are now ready to prove \cref{prop:scalar-mean_strict}.
We proceed in a similar way as in~\cite{schoen-yau-psc-manifolds}.
Under the hypotheses of \cref{prop:scalar-mean_strict}, we use the spectral information from \cref{lem:positive_eigenvalues} to make a conformal change of the metric on (a suitable component of) the three-dimensional $\mu$\nobreakdash-bubble in such a way that the new metric would contradict \cref{prop:Listing_3D} if $f$ had non\nobreakdash-zero degree.

\begin{proof}[Proof of \cref{prop:scalar-mean_strict}]
    Suppose, by contradiction, that $f$ has non\nobreakdash-zero degree.
    Using \cref{lem:f_Sigma_degree,lem:separating_hypersurface,lem:positive_eigenvalues}, we choose a closed connected oriented hypersurface $Y$ embedded in $V$ such that
    \begin{itemize}[label=$\vartriangleright$]
        \item $f_Y\colon Y\to\sph^3$ has non\nobreakdash-zero degree, where $f_Y$ is the function defined by~\eqref{eq:f_Z};
        \item the lowest  eigenvalue $\lambda$ of $\mathcal L_{g_Y}$ is positive, where $g_Y$ denotes the restriction of $g$ to $Y$, and where $\mathcal L_{g_Y}$ is the operator defined by~\eqref{eq:L_g}.
    \end{itemize}
    Let $u$ be an eigenfunction relative to $\lambda$, that is, $0\not\equiv u\in C^\infty(Y)$ and satisfies
    \begin{equation}\label{eq:conformal1}
        -\Delta_{g_Y}u+\frac{1}{8}\scal_{g_Y}u
        =\frac{3}{4}\left\|d f_Y\right\|^2u+\lambda u.
    \end{equation}
    By classical elliptic theory, $u$ doesn't change sign.
    We can and we will assume that $u$ is strictly positive.
    Consider the conformal metric
    \begin{equation*}
        \bar g_Y\coloneqq u^4\cdot g_Y.
    \end{equation*}
    The classical formula for the scalar curvature under conformal change~\cite[Corollary~1.161]{besse-einstein-manifolds} gives
    \begin{equation}\label{eq:conformal2}
        \scal_{\bar g_Y}=8\cdot u^{-5}\left(-\Delta_{g_Y}u+\frac{1}{8}\scal_{g_Y} u\right).
    \end{equation}
    We now compare the operator norms of $d f_Y$ with respect to the metrics $g_Y$ and $\bar g_Y$.
    Correspondingly, we use the notation $\left\|df_Y\right\|_{g_Y}$ and $\left\|df_Y\right\|_{\bar g_Y}$.
    A direct calculation shows that
    \begin{equation}\label{eq:conformal3}
        \left\|df_Y\right\|^2_{\bar g_Y}=u^{-4}\left\|df_Y\right\|^2_{g_Y}.
    \end{equation}
    Since $u$ is postive, from \cref{eq:conformal1,eq:conformal2,eq:conformal3} we deduce
    \begin{align*}
        \scal_{\bar g_Y}=6u^{-4}\left\|df_Y\right\|^2_{g_Y}+8\lambda u^{-4}
        >6u^{-4}\left\|df_Y\right\|^2_{g_Y}=6\left\|df_Y\right\|^2_{\bar g_Y}.
    \end{align*}
    Hence, we constructed a three-dimensional closed oriented Riemannian manifold $(Y,\bar g_Y)$ and a smooth map $f_Y\colon (Y,\bar g_Y)\to(\sph^3,g_{\sph^3})$ of non\nobreakdash-zero degree satisfying $\scal_{\bar g_Y}>6\left\|df_Y\right\|^2_{\bar g_Y}$.
    By \cref{rem:strict-Listing-3D}, this contradicts \cref{prop:Listing_3D}.
\end{proof}

\section{Ricci flow, harmonic map heat flow, and Einstein metrics}\label{sec:Ricci_flow_Einstein}
We employ estimates by Lee and Tam~\cite{LeeTam} on the harmonic map heat flow coupled with the Ricci flow to demonstrate that, under the hypotheses of \cref{thm:sphere}, if the metric $g$ is non-Einstein, then there exists a metric $\tilde{g}$ and a function $\tilde{f}$ satisfying the hypotheses of \cref{thm:sphere} with all inequalities being strict.
For a comprehensive overview of the Ricci flow and the harmonic map heat flow, we recommend referring to~\cite{topping-lectures-ricci-flow} and~\cite{lin-wang-harmonic-maps}, respectively. 
Recall that a metric $g$ is Einstein if $\Ric_g=c g$ for some constant $c$, known as the proportionality constant of $g$.
For a smooth map $f\colon(M,g)\to(N,h)$ of Riemannian manifolds, we denote by $\Lip(f)$ the Lipschitz constant of $f$.

\begin{proposition}\label{prop:Ricci_bump_up}
Let $(M,g)$ be an $n$-dimensional closed Riemannian manifold with $\scal_g\geq n(n-1)$, and let $f\colon (M,g)\to(\sph^n,g_{\sph^n})$ be a smooth, distance non\nobreakdash-increasing map of non-zero degree. If the metric $g$ is non-Einstein, then there exists a Riemannian metric $\tilde{g}$ on $M$ and a smooth map $\tilde{f}\colon (M,\tilde{g})\to(\sph^n,g_{\sph^n})$ of non-zero degree such that $\scal_{\tilde{g}}>n(n-1)$ and $\Lip(\tilde{f})<1$.
\end{proposition}

Recall that a Ricci flow on a manifold $M$ is a smooth family of Riemannian metrics $(g_t)_{t\in[0,T]}$ on $M$ satisfying the Ricci equation
\begin{equation*}
    \partial_tg_t=-2\Ric_{g_t}.
\end{equation*}
The short-time existence and primary properties of the Ricci flow were established by Hamilton in his seminal work~\cite{hamilton-three-manifolds}. 
The following lemma summarizes well\nobreakdash-known properties of the Ricci flow utilized in this paper. 
We include a brief proof for clarity.

\begin{lemma}\label{lem:Ricci_lower_bound}
    Let $(M,g)$ be an $n$\nobreakdash-dimensional closed Riemannian manifold with $\scal_g\geq n(n-1)$.
    If $(g_t)_{t\in[0,T]}$ is a Ricci flow on $M$ such that $g_0=g$, then
    \begin{equation}\label{eq:Ricci_lower_bound}
        \scal_{g_t} \geq \frac{n(n-1)}{1-2(n-1)t},\qquad\forall t\in[0,T].
    \end{equation}
    Moreover, the previous inequality is strict for $t\in(0,T]$, unless $g$ is an Einstein metric with $\Ric_g=(n-1)g$.
\end{lemma}

\begin{proof}
Recall the orthogonal decomposition
\[
    \Ric_{g_t}=\mathring\Ric_{g_t}+\frac{\scal_{g_t}}{n}g_t,
\]
where $\mathring\Ric_{g_t}$ is the traceless component of the Ricci tensor of $g_t$.
It is well-known that the scalar curvature $\scal_{g_t}$ evolves according to the equation
\begin{equation}\label{eq:Ricci_lower_bound1}
    \left(\partial_t-\Delta_{g_t}\right)\scal_{g_t}
    =2|\Ric_{g_t}|^2
    =2|\mathring\Ric_{g_t}|^2+2\frac{\scal_{g_t}^2}{n}
    \geq 2\frac{\scal_{g_t}^2}{n},
\end{equation}
see~\cite[Section~2.5]{topping-lectures-ricci-flow}.
Here, $\Delta_{g_t}$ denotes the Laplace-Beltrami operator of $(M,g_t)$.
From~\eqref{eq:Ricci_lower_bound1}, $\scal_{g_t}$ satisfies
\[
        \left(\partial_t-\Delta_{g_t}\right)\scal_{g_t}\geq2\frac{\scal_{g_t}^2}{n}.
\]
The function $\psi(t)=n(n-1)/(1-2(n-1)t)$ solves
\[
    \begin{cases}
    \psi^\prime(t)=2\frac{\psi^2(t)}{n}\\
    \psi(0)=n(n-1).
    \end{cases}
\]
Since $\scal_{g_0}=\scal_g\geq n(n-1)=\psi(0)$, from the weak minimum principle~\cite[Theorem 3.1.1]{topping-lectures-ricci-flow} we deduce Inequality~\eqref{eq:Ricci_lower_bound}.
To prove the last assertion, suppose Inequality~\eqref{eq:Ricci_lower_bound} is an equality for some $(p,T_0)\in M\times(0,T]$.
By the strong minimum principle, 
\[
    \scal_{g_t}=\psi(t)=\frac{n(n-1)}{1-2(n-1)t},\qquad\forall t\in[0,T_0].
\]
See for example \cite[Theorem 2.1.1]{mantegazza-mean-curvature-flow}.
Hence, $\partial_t\scal_{g_t}=2\scal_{g_t}^2/n$, $\Delta_{g_t}\scal_{g_t}=0$, and the inequality in~\eqref{eq:Ricci_lower_bound1} is an equality.
Thus,
\[
    0=\mathring\Ric_{g_t}=\Ric_{g_t}-\frac{\scal_{g_t}}{n}g_t=\Ric_{g_t}-\frac{\psi(t)}{n}g_t,\qquad \forall t\in[0,T_0].
\]
This shows that $g_t$ is an Einstein metric with proportionality constant $\scal_{g_t}/n=\psi(t)/n$ for every $t\in[0,T_0]$.
In particular, $g=g_0$ satisfies $\Ric_g=(n-1)g$.
\end{proof}

By utilizing estimates of Lee and Tam~\cite{LeeTam} for the harmonic map heat flow coupled with the Ricci flow, we derive the following lemma.

\begin{lemma}\label{lem:harmonic_heat_flow}
    Let $(M,g)$ be an $n$-dimensional closed Riemannian manifold, and let $f: (M,g) \to (\sph^n,g_{\sph^n})$ be a smooth, distance non\nobreakdash-increasing map.
    Let $(g_t)_{t\in[0,T]}$ be a Ricci flow on $M$ such that $g_0=g$.
    Then there exists $T_1\in (0,T]$ and a smooth family of smooth maps $f_t\colon(M,g_t)\to(\sph^n,g_{\sph^n})$, for $t\in[0,T_1]$, such that $f_0=f$ and
    \begin{equation}\label{eq:lipschitz-f_t}
        \Lip(f_t)\leq\frac{1}{1-2(n-1)t},\qquad\forall t \in [0,T_1].
    \end{equation}
\end{lemma}

\begin{proof}
By~\cite[Theorem~1.1]{huang-tam-short-time-existence}, there exists $T_1\in (0,T]$ and a smooth family of smooth maps $f_t\colon(M,g_t)\to(\sph^n,g_{\sph^n})$, for $t\in[0,T_1]$, such that
\[
    \begin{cases}
        \partial_tf_t=\tau(f_t)\\
        f_0=f
    \end{cases}
\]
where $\tau(f_t)$ is the tension field of the map $f_t\colon(M,g_t)\to(\sph^n,g_{\sph^n})$.
For the definition of the tension field, we refer to~\cite[Section~4.1]{huang-tam-short-time-existence}.
By applying \cite[Theorem~2.1]{LeeTam} with $k=-\Ric_{g_t}$ and $\kappa=1$, we conclude that each $f_t$ satisfies~\eqref{eq:lipschitz-f_t}.
\end{proof}

\begin{proof}[Proof of \cref{prop:Ricci_bump_up}]
Suppose $g$ is not an Einstein metric with proportionality constant $(n-1)$.
Let $(g_t)_{t\in[0,T]}$ be a Ricci flow on $M$ such that $g_0=g$.
By \cref{lem:Ricci_lower_bound},
\[
    \scal_{g_t}>\frac{n(n-1)}{1-2(n-1)t},\qquad\forall t\in(0,T].
\]
By \cref{lem:harmonic_heat_flow}, there exist $T_1\in(0,T]$ and a smooth family of smooth maps $f_t\colon(M,g_t)\to(\sph^n,g_{\sph^n})$, for $t\in[0,T_1]$, such that
\[
    \Lip(f_t)\leq\frac{1}{1-2(n-1)t},\qquad\forall t \in [0,T_1].
\]
Since $f$ has non\nobreakdash-zero degree, each $f_t$ has non\nobreakdash-zero degree as well.
Let $T_2\in(0,T_1]$ be fixed.
Let $\epsilon>0$ be such that $c_\epsilon\coloneqq 1-2(n-1)T_2-\epsilon>0$.
By taking $\epsilon$ sufficiently small, the metric $c_\epsilon^{-1} g_{T_2}$ possesses the desired properties.
\end{proof}

\section{Rigidity of the four-dimensional sphere}\label{sec:4-sphere_rigidity}

We combine the results from \cref{sec:scalar_mean,sec:Ricci_flow_Einstein} to establish the rigidity properties stated in \cref{thm:sphere}.
First, we utilize \cref{prop:scalar-mean_strict} to address the case where all the inequalities in \cref{thm:sphere} are strict.

\begin{lemma}\label{lem:strict_inequalities}
    Let $(M,g)$ be a four-dimensional closed oriented Riemannian manifold with $\scal_g>12$.
    If $f\colon (M,g)\to(\sph^4,g_{\sph^4})$ is a strictly distance-decreasing smooth map,
    then $f$ has degree zero.
\end{lemma}

Before presenting the proof, we introduce some additional notation. 
For $\ell\in(0,\pi/2)$, consider the spherical band 
\[
    (\sph^3_\ell,g_{\cos})=\left(\sph^3\times[-\ell,\ell],\cos^2(t)g_{\sph^3}+dt^2\right).
\]
Let $p_\pm$ be antipodal points in $\sph^4$.
Utilizing the identification $(\sph^4\setminus\{p_\pm\},g_{\sph^4})\cong(\sph^3\times(-\pi/2,\pi/2),\cos^2(t)g_{\sph^3}+dt^2)$, we view $(\sph^3_\ell,g_{\cos})$ as a Riemannian band isometrically embedded in $(\sph^4\setminus\{p_\pm\},g_{\sph^4})$.

\begin{proof}[Proof of \cref{lem:strict_inequalities}]
    Suppose, for the sake of contradiction, that $f$ has non\nobreakdash-zero degree.
    Let $\delta\in(0,1)$ be such that $\Lip(f)\leq 1-\delta$.
    Since the antipodal map on $\sph^4$ is an isometry, using the Brown-Sard theorem, we choose two antipodal points $p_\pm$ in $\sph^4$ that are regular values of $f$.
    For $\ell\in (0,\pi/2)$, we consider the spherical band $(\sph^3_\ell,g_{\cos})$ isometrically embedded in $(\sph^4\setminus\{p_\pm\},g_{\sph^4})$.
    Since $M$ is compact, $f^{-1}(\{\pm p\})$ is a finite set.
    Since $p_+$ is a regular value of $f$, we choose an open neighborhood $U_+$ of $p_+$ such that each component of $f^{-1}(U_+)$ contains only one point in $f^{-1}(p_+)$ and $f$ is a diffeomorphism when restricted to every component of $f^{-1}(U_+)$.
    We choose an open neighborhood $U_-$ of $p_-$ with similar properties as $U_+$, such that $U_-\cap U_+=\emptyset$.
    Additionally, we choose $L\in (\pi/(2+2\delta),\pi/2)$ such that $\partial_\pm\sph^3_L\subset U_\pm$.
    With this choice, $V=f^{-1}(\sph^3_L)$ is a four-dimensional oriented band with $f^{-1}(\partial_\pm\sph^3_L)=\partial_\pm V$.
    Furthermore, the restriction $f_V=f|_V\colon V\to\sph^3_L$ is a smooth map of non\nobreakdash-zero degree.
    We choose $\ell\in (L,\pi/2)$ such that
    \begin{equation*}
        \mean(\partial_\pm V)> -3\tan(\ell)=H_{g_\phi}(\partial_\pm\sph^3_\ell).
    \end{equation*}
    Let $h_\ell\colon (\sph^3_L,g_{\cos})\to (\sph^3_\ell,g_{\cos})$ be the smooth map defined as 
    \begin{equation*}
        h_\ell(x,t)=(x,t\ell/L)
    \end{equation*}
    for $(x,t)\in\sph^3\times[-\ell,\ell]=\sph^3_\ell$.
    Note that $\Lip(h_\ell)<\pi/(2L)$.
    Define $f_\ell\coloneqq h_\ell\circ f_V$.
    Since $\Lip(h_\ell)<\pi/(2L)$ and $L>\pi/(2+2\delta)$, then $\Lip(f_\ell)<(1-\delta)\pi/(2L)<1$.
    Finally, $f_\ell$ has non\nobreakdash-zero degree, since it is the composition of maps of non\nobreakdash-zero degree.    
    Since $\scal_g>12$ and $\cos(t)$ is log-concave, this leads to a contradiction with \cref{prop:scalar-mean_strict}.
\end{proof}

Next, we establish the scalar curvature rigidity of the $n$-sphere for Einstein manifolds.

\begin{lemma}\label{lem:Einstein_Llarull}
    Let $(M,g)$ be an $n$-dimensional closed connected oriented Einstein manifold with $\Ric_g=(n-1)g$.
    If $f\colon(M,g)\to (\sph^n,g_{\sph^n})$ is a smooth distance non\nobreakdash-increasing map of non\nobreakdash-zero degree,
    then $f$ is an isometry.
\end{lemma}

\begin{proof}
    By degree theory,
    \begin{equation}\label{eq:Einstein_rigidity1}
        \deg(f)\int_{\sph^n}\,dV_{g_{\sph^n}}=\int_Mf^\ast(dV_{g_{\sph^n}}).
    \end{equation}
    Note that $f^\ast(dV_{g_{\sph^n}})=\det(df)\,dV_g$, where $\det(df_p)$ is the determinant of $df_p\colon T_pM\to T_{f(p)}\sph^n$ as linear map of oriented inner product spaces.
    From \eqref{eq:Einstein_rigidity1},
    \begin{equation}\label{eq:Einstein_rigidity}
        \left|\deg(f)\right|\vol(\sph^n,g_{\sph^n})
        \leq \int_M\left|\det(df)\right|\,dV_g.
    \end{equation}
    Let $p\in M$.
    Since $f$ is distance non-increasing, $\left|\det(df_p)\right|\leq1$.
    Since $\Ric_g=(n-1)g$, by Bishop's volume comparison~\cite[Section~9.1.1, Lemma~35]{Petersen_Riemannian_geometry} $\vol(M,g)\leq\vol(\sph^n,g_{\sph^n})$.
    Hence, from~\eqref{eq:Einstein_rigidity} we deduce
    \begin{equation*}
        \left|\deg(f)\right|\vol(\sph^n,g_{\sph^n})
        \leq \int_M\left|\det(df)\right|\,dV_g
        \leq\vol(M,g)
        \leq\vol(\sph^n,g_{\sph^n}).
    \end{equation*}
    Since $\deg(f)\neq0$, it follows that $\left|\deg(f)\right|=1$ and all inequalities must be equalities.
    Thus,
    \begin{equation*}
        \int_M\left|\det(df)\right|\,dV_g
        =\vol(M,g).
    \end{equation*}
    Therefore, $\left|\det(df_p)\right|=1$ for every $p\in M$.
    Since $f$ is distance non\nobreakdash-increasing, we conclude that $f$ is a local isometry.
    Since $\sph^n$ is simply-connected, $f$ is an isometry.
\end{proof}

We are now in the position to prove our main theorem.

\begin{proof}[Proof of \cref{thm:sphere}]
    Let us first show that $g$ must be Einstein with $\Ric_g=3g$.
    Suppose, by contradiction, that this is not the case. 
    By \cref{prop:Ricci_bump_up}, there exists a Riemannian metric $\tilde g$ on $M$ and a smooth map $\tilde f\colon (M,\tilde g)\to(\sph^4,g_{\sph^4})$ of non\nobreakdash-zero degree such that $\scal_{\tilde g}>12$ and $\Lip(f)<1$, contradicting \cref{lem:strict_inequalities}.

    We conclude that $g$ is Einstein with $\Ric_g=3 g$.
    By \cref{lem:Einstein_Llarull}, $f$ is an isometry.
\end{proof}

\bibliographystyle{amsalpha}
\bibliography{Llarull4D}

\begin{bibdiv}
\begin{biblist}

\bib{besse-einstein-manifolds}{book}{
      author={Besse, Arthur~L.},
       title={Einstein manifolds},
      series={Classics in Mathematics},
   publisher={Springer-Verlag, Berlin},
        date={2008},
        ISBN={978-3-540-74120-6},
        note={Reprint of the 1987 edition},
      review={\MR{2371700}},
}

\bib{Baer:2023aa}{article}{
      author={Bär, Christian},
      author={Brendle, Simon},
      author={Hanke, Bernhard},
      author={Wang, Yipeng},
       title={Scalar curvature rigidity of warped product metrics},
        date={2023},
      eprint={2306.04015},
         url={https://arxiv.org/pdf/2306.04015.pdf},
}

\bib{cecchini-two-ends}{article}{
      author={Cecchini, Simone},
      author={R\"{a}de, Daniel},
      author={Zeidler, Rudolf},
       title={Nonnegative scalar curvature on manifolds with at least two ends},
        date={2023},
        ISSN={1753-8416},
     journal={J. Topol.},
      volume={16},
      number={3},
       pages={855\ndash 876},
         url={https://doi.org/10.1112/topo.12303},
      review={\MR{4611408}},
}

\bib{cecchini2022scalar}{article}{
      author={Cecchini, Simone},
      author={Zeidler, Rudolf},
       title={Scalar and mean curvature comparison via the dirac operator},
        date={2022},
      eprint={2103.06833},
}

\bib{chodosh2023generalized}{article}{
      author={Chodosh, Otis},
      author={Li, Chao},
       title={Generalized soap bubbles and the topology of manifolds with positive scalar curvature},
        date={2023},
      eprint={2008.11888},
}

\bib{chodosh-sufficiently-connected}{article}{
      author={Chodosh, Otis},
      author={Li, Chao},
      author={Liokumovich, Yevgeny},
       title={Classifying sufficiently connected {PSC} manifolds in 4 and 5 dimensions},
        date={2023},
        ISSN={1465-3060},
     journal={Geom. Topol.},
      volume={27},
      number={4},
       pages={1635\ndash 1655},
         url={https://doi.org/10.2140/gt.2023.27.1635},
      review={\MR{4602422}},
}

\bib{giusti-minimal-surfaces}{book}{
      author={Giusti, Enrico},
       title={Minimal surfaces and functions of bounded variation},
      series={Monographs in Mathematics},
   publisher={Birkh\"{a}user Verlag, Basel},
        date={1984},
      volume={80},
        ISBN={0-8176-3153-4},
         url={https://doi.org/10.1007/978-1-4684-9486-0},
      review={\MR{775682}},
}

\bib{Gromov-macroscopic-dimension}{incollection}{
      author={Gromov, M.},
       title={Positive curvature, macroscopic dimension, spectral gaps and higher signatures},
        date={1996},
   booktitle={Functional analysis on the eve of the 21st century, {V}ol. {II} ({N}ew {B}runswick, {NJ}, 1993)},
      series={Progr. Math.},
      volume={132},
   publisher={Birkh\"{a}user Boston, Boston, MA},
       pages={1\ndash 213},
         url={https://doi.org/10.1007/s10107-010-0354-x},
      review={\MR{1389019}},
}

\bib{gromov2020metrics}{article}{
      author={Gromov, Misha},
       title={No metrics with positive scalar curvatures on aspherical 5-manifolds},
        date={2020},
      eprint={2009.05332},
}

\bib{gromov-four-lectures}{incollection}{
      author={Gromov, Misha},
       title={Four lectures on scalar curvature},
        date={[2023] \copyright 2023},
   booktitle={Perspectives in scalar curvature. {V}ol. 1},
   publisher={World Sci. Publ., Hackensack, NJ},
       pages={1\ndash 514},
      review={\MR{4577903}},
}

\bib{hamilton-three-manifolds}{article}{
      author={Hamilton, Richard~S.},
       title={Three-manifolds with positive {R}icci curvature},
        date={1982},
        ISSN={0022-040X},
     journal={J. Differential Geometry},
      volume={17},
      number={2},
       pages={255\ndash 306},
         url={http://projecteuclid.org/euclid.jdg/1214436922},
      review={\MR{664497}},
}

\bib{hu-3d-spherical}{article}{
      author={Hu, Yuhao},
      author={Liu, Peng},
      author={Shi, Yuguang},
       title={Rigidity of 3{D} spherical caps via {$\mu$}-bubbles},
        date={2023},
        ISSN={0030-8730},
     journal={Pacific J. Math.},
      volume={323},
      number={1},
       pages={89\ndash 114},
         url={https://doi.org/10.2140/pjm.2023.323.89},
      review={\MR{4594778}},
}

\bib{huang-tam-short-time-existence}{article}{
      author={Huang, Shaochuang},
      author={Tam, Luen-Fai},
       title={Short-time existence for harmonic map heat flow with time-dependent metrics},
        date={2022},
        ISSN={1050-6926},
     journal={J. Geom. Anal.},
      volume={32},
      number={12},
       pages={Paper No. 287, 32},
         url={https://doi.org/10.1007/s12220-022-01035-6},
      review={\MR{4487751}},
}

\bib{spingeometry}{book}{
      author={Lawson, H.~Blaine, Jr.},
      author={Michelsohn, Marie-Louise},
       title={Spin geometry},
      series={Princeton Mathematical Series},
   publisher={Princeton University Press, Princeton, NJ},
        date={1989},
      volume={38},
        ISBN={0-691-08542-0},
      review={\MR{1031992}},
}

\bib{LeeTam}{article}{
      author={Lee, Man-Chun},
      author={Tam, Luen-Fai},
       title={Rigidity of {L}ipschitz map using harmonic map heat flow},
        date={2022},
      eprint={2207.11017},
}

\bib{li-polyhedron}{article}{
      author={Li, Chao},
       title={A polyhedron comparison theorem for 3-manifolds with positive scalar curvature},
        date={2020},
        ISSN={0020-9910},
     journal={Invent. Math.},
      volume={219},
      number={1},
       pages={1\ndash 37},
         url={https://doi.org/10.1007/s00222-019-00895-0},
      review={\MR{4050100}},
}

\bib{lin-wang-harmonic-maps}{book}{
      author={Lin, Fanghua},
      author={Wang, Changyou},
       title={The analysis of harmonic maps and their heat flows},
   publisher={World Scientific Publishing Co. Pte. Ltd., Hackensack, NJ},
        date={2008},
        ISBN={978-981-277-952-6; 981-277-952-3},
         url={https://doi.org/10.1142/9789812779533},
      review={\MR{2431658}},
}

\bib{Listing:2010te}{article}{
      author={Listing, Mario},
       title={Scalar curvature on compact symmetric spaces},
        date={2010},
      eprint={1007.1832},
         url={https://arxiv.org/pdf/1007.1832.pdf},
}

\bib{Llarull}{article}{
      author={Llarull, Marcelo},
       title={Sharp estimates and the {D}irac operator},
        date={1998},
        ISSN={0025-5831},
     journal={Math. Ann.},
      volume={310},
      number={1},
       pages={55\ndash 71},
         url={https://doi.org/10.1007/s002080050136},
      review={\MR{1600027}},
}

\bib{mantegazza-mean-curvature-flow}{book}{
      author={Mantegazza, Carlo},
       title={Lecture notes on mean curvature flow},
      series={Progress in Mathematics},
   publisher={Birkh\"{a}user/Springer Basel AG, Basel},
        date={2011},
      volume={290},
        ISBN={978-3-0348-0144-7},
         url={https://doi.org/10.1007/978-3-0348-0145-4},
      review={\MR{2815949}},
}

\bib{Petersen_Riemannian_geometry}{book}{
      author={Petersen, Peter},
       title={Riemannian geometry},
     edition={Third},
      series={Graduate Texts in Mathematics},
   publisher={Springer, Cham},
        date={2016},
      volume={171},
        ISBN={978-3-319-26652-7; 978-3-319-26654-1},
         url={https://doi.org/10.1007/978-3-319-26654-1},
      review={\MR{3469435}},
}

\bib{Raede23}{article}{
      author={Räde, Daniel},
       title={Scalar and mean curvature comparison via {$\mu $}-bubbles},
        date={2023},
        ISSN={0944-2669},
     journal={Calc. Var. Partial Differential Equations},
      volume={62},
      number={7},
       pages={Paper No. 187, 39},
         url={https://doi.org/10.1007/s00526-023-02520-8},
      review={\MR{4612761}},
}

\bib{schoen-yau-psc-manifolds}{article}{
      author={Schoen, R.},
      author={Yau, S.~T.},
       title={On the structure of manifolds with positive scalar curvature},
        date={1979},
        ISSN={0025-2611},
     journal={Manuscripta Math.},
      volume={28},
      number={1-3},
       pages={159\ndash 183},
         url={https://doi.org/10.1007/BF01647970},
      review={\MR{535700}},
}

\bib{topping-lectures-ricci-flow}{book}{
      author={Topping, Peter},
       title={Lectures on the {R}icci flow},
      series={London Mathematical Society Lecture Note Series},
   publisher={Cambridge University Press, Cambridge},
        date={2006},
      volume={325},
        ISBN={978-0-521-68947-2; 0-521-68947-3},
         url={https://doi.org/10.1017/CBO9780511721465},
      review={\MR{2265040}},
}

\bib{wang2023scalar}{article}{
      author={Wang, Jinmin},
      author={Xie, Zhizhang},
       title={Scalar curvature rigidity of degenerate warped product spaces},
        date={2023},
      eprint={2306.05413},
}

\bib{zhu_width_mu}{article}{
      author={Zhu, Jintian},
       title={Width estimate and doubly warped product},
        date={2021},
        ISSN={0002-9947,1088-6850},
     journal={Trans. Amer. Math. Soc.},
      volume={374},
      number={2},
       pages={1497\ndash 1511},
         url={https://doi.org/10.1090/tran/8263},
      review={\MR{4196400}},
}

\end{biblist}
\end{bibdiv}

\end{document}